\documentclass[a4paper,12pt]{article}

\usepackage[left=2cm,right=2cm, top=2cm,bottom=2cm,bindingoffset=0cm]{geometry}

\usepackage{verbatim}
\usepackage{amsmath}
\usepackage{amsthm}
\usepackage{amssymb}
\usepackage{delarray}
\usepackage{cite}
\usepackage{hyperref}
\usepackage{mathrsfs}
\usepackage{graphicx}
\usepackage{tikz}
\usetikzlibrary{patterns}
\usepackage{caption}
\DeclareCaptionLabelSeparator{dot}{. }
\newcommand{\al}{\alpha}
\newcommand{\be}{\beta}
\newcommand{\ga}{\gamma}

\newcommand{\la}{\lambda}
\newcommand{\om}{\omega}

\newcommand{\eps}{\varepsilon}

\newcommand{\iy}{\infty}
\theoremstyle{plain}

\newtheorem{thm}{Theorem}
\newtheorem{lem}{Lemma}

\theoremstyle{definition}

\newtheorem{alg}{Algorithm}
\newtheorem{ip}{Inverse Problem}

\theoremstyle{remark}

 \allowdisplaybreaks

\begin{document}

\begin{center}
{\large\bf A partial inverse problem for the Sturm-Liouville operator on the graph with a loop	
}
\\[0.2cm]
{\bf Chuan-Fu Yang, Natalia P. Bondarenko} \\[0.2cm]
\end{center}

\vspace{0.5cm}

{\bf Abstract.} The Sturm-Liouville operator with singular potentials on the lasso graph is considered.
We suppose that the potential is known a priori on the boundary edge, and recover the potential on the loop
from a part of the spectrum and some additional data. We prove the uniqueness theorem and provide a constructive algorithm
for the solution of this partial inverse problem.
 
\medskip

{\bf Keywords:} partial inverse spectral problem; Sturm-Liouville operator on a lasso graph; quantum graph; singular potential.

\medskip

{\bf AMS Mathematics Subject Classification (2010):} 34A55; 34B05; 34B09; 34B45; 34L20; 34L40; 47E05

\vspace{1cm}

{\large \bf 1. Introduction}

\bigskip

Differential operators on geometrical graphs, also called quantum graphs, are intensively studied by mathematicians in recent years,
and have applications in organic chemistry, mechanics, mesoscopic physics, nanotechnology, theory of waveguides and other branches
of science \cite{Kuch02}. There is an extensive literature, devoted to quantum graphs. We mention here works \cite{Exner08, PPP04},
where the further references could be found. A nice elementary introduction to the theory of quantum graphs is provided in \cite{Kuch04}.
The recent paper \cite{Yur16} contains a good overview of results on inverse spectral problems for differential operators on graphs,
which consist in recovering differential operators (especially coefficients of differential expressions) from various types of spectral data. 

In this paper, we consider the Sturm-Liouville operator on the graph with a loop. We suppose that the potential is known a priori on a part of the graph,
and recover the potential on the remaining part from a part of the spectrum and some additional data. 
Such partial inverse problems have been studied in papers
\cite{Piv00, Yang10, Yang11, Yur09, Bond17-1, Bond17-2, Bond17-preprint} for star-shaped graphs. 
In the present paper, we obtain the first results
in this direction for the graph with a loop. We formulate a partial inverse problem on a lasso graph (see Figure~\ref{img:lasso}), 
prove the uniqueness theorem and provide a constructive algorithm for solution of this problem. 
Note that complete inverse problems for differential operators on lasso graphs were studied in \cite{MMT08, MT12, Kur13}.
We hope that in the future, our results will be generalized for graphs with a more complicated structure.
We develop the technique of \cite{Bond17-1, Bond17-2, Bond17-preprint}, based on the Riesz basis property of some systems of vector functions. 
We also mention that our problem on a graph is related to the Hochstadt-Lieberman problem on a finite interval \cite{HL78}.

The paper is organized as follows. In Section~2, we state the boundary value problem on the lasso graph and study asymptotic properties
of its eigenvalues. Section~3 is devoted to the periodic inverse Sturm-Liouville problem, which is further used as an auxiliary step for recovering 
the potential on the loop. In Section~4, we formulate the partial inverse problem, provide our main results and proofs.

\bigskip

{\large \bf 2. Asymptotic formulas for eigenvalues}

\bigskip

Consider the lasso graph $G$, represented in Figure~\ref{img:lasso}. The edge $e_1$ is a boundary edge of length $l_1 = m \in \mathbb N$, 
the edge $e_2$ is a loop of length $l_2 = 1$. Introduce a parameter $x_j$ for each edge $e_j$, $j = 1, 2$, $x_j \in [0, l_j]$.
The value $x_1 = 0$ corresponds to the boundary vertex, and $x_1 = m$ corresponds to the internal vertex. 
For the loop $e_2$, both ends $x_2 = 0$ and $x_2 = 1$ correspond to the internal vertex.

\begin{figure}[h!]
\centering
\begin{tikzpicture}
\filldraw (2, 1) circle (2pt);
\filldraw (7, 1) circle (2pt);
\draw (2, 1) edge node[auto]{$e_1$} (7, 1);
\draw (1, 1) circle [radius = 1] node[auto]{$e_2$};
\draw (2.3, 1.3) node{$l_1$};
\draw (6.7, 1.3) node {$0$};
\draw (1.7, 1.3) node {$0$};
\draw (1.7, 0.7) node {$l_2$};
\end{tikzpicture}
\caption{Lasso graph}
\label{img:lasso}
\end{figure}
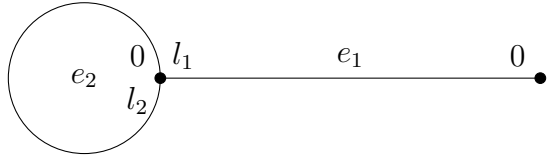

Let $y = [y_j(x_j)]_{j = 1, 2}$ be a vector function on the graph $G$.
Consider Sturm-Liouville expressions
$$
\ell_j y_j := -y_j'' + q_j(x_j) y_j, \quad j = 1, 2,  
$$
on the edges of $G$, where $q_j$, $j = 1, 2$, are real-valued functions from 
$W_2^{-1}(0, l_j)$. This means that $q_j = \sigma_j'$, $\sigma_j \in L_2(0, l_j)$, where the derivative is understood in the sense of distributions. 
We call the functions $\sigma_j$ the {\it potentials}. 
Define the {\it quasi-derivatives} $y_j^{[1]} = y_j' - \sigma_j y_j$, $j = 1, 2$.
Then the differential expressions $\ell_j$ can be undrestood in the following sense:
$$
\ell_j y_j = -(y_j^{[1]})' - \sigma_j(x_j) y_j^{[1]} - \sigma_j^2(x_j) y_j,
$$
on the domain 
$$
\mathcal D(\ell_j) = \{ y_j \in W_2^1[0, l_j] \colon y_j^{[1]} \in W_1^1[0, l_j], \: \ell_j y_j \in L_2(0, l_j) \}. 
$$

Inverse problems for Sturm-Liouville operators with singular potentials $q \in W_2^{-1}$ on a finite interval
were studied in papers \cite{HM03, HM04-2spectra, HM04-half}. However, there are only a few results for such operators
on graphs (see \cite{FIY08, Bond17-preprint}). For the purposes of the present paper, the more general class $W_2^{-1}$ causes
no additional difficulties in comparison with $L_2$.

We study the boundary value problem $L$ for the Sturm-Liouville equations on the graph $G$:
\begin{equation} \label{eqv}
(\ell_j y_j)(x_j) = \la y_j(x_j), \quad x_j \in (0, l_j), \quad j = 1, 2,
\end{equation}
with the standard matching conditions
$$
   y_1(m) = y_2(0) = y_2(1), \quad y_1^{[1]}(m) - y_2^{[1]}(0) + y_2^{[1]}(1) = 0
$$
in the internal vertex, and the Dirichlet boundary condition $y_1(0) = 0$ in the boundary vertex.

For each fixed $j = 1, 2$, let $C_j(x_j, \la)$ and $S_j(x_j, \la)$ be the solutions of the corresponding equation \eqref{eqv}
under the initial conditions 
$$
   C_j(0, \la) = S_j^{[1]}(0, \la) = 1, \quad C_j^{[1]}(0, \la) = S_j(0, \la) = 0.
$$

Further we use the following notations.
Let $B_{2, a}$ be the class of Paley-Wiener functions of exponential type not greater than $a$, belonging to $L_2(\mathbb R)$.
The symbols $\varkappa_{k, odd}(\rho)$ and $\varkappa_{k, even}(\rho)$ denote various odd and even functions from $B_{2, k}$, respectively.
Note that
$$
    \varkappa_{k, odd}(\rho) = \int_0^k \mathcal K(t) \sin \rho t \, dt, \quad \varkappa_{k, even}(\rho) = \int_0^k \mathcal N(t) \cos \rho t \,dt,
$$
where $\mathcal K, \mathcal N \in L_2(0, k)$.
The notation $\{ \varkappa_n \}$ stands for various sequences in $l_2$.

Relying on the results of papers \cite{HM03, HM04-transform, HM04-2spectra}, we obtain 
the following relations for $j = 1, 2$:
\begin{equation} \label{CS}
\arraycolsep=1.6pt\def\arraystretch{1.5}
\left.
\begin{array}{ll}
   C_j(l_j, \la) & = \cos \rho l_j + \varkappa_{l_j, even}(\rho), \\
   S_j(l_j, \la) & = \dfrac{\sin \rho l_j}{\rho} + \dfrac{\varkappa_{l_j, odd}(\rho)}{\rho}, \\
   S_j^{[1]}(l_j, \la) & = \cos \rho l_j + \varkappa_{l_j, even}(\rho).
\end{array}
\qquad \right\}
\end{equation}

The boundary value problem $L$ has a purely discrete spectrum, consisting of real eigenvalues.
The eigenvalues of $L$ coincide with the zeros of the characteristic function
\begin{equation} \label{Delta}
   \Delta(\la) = S_1^{[1]}(m, \la) S_2(1, \la) + S_1(m, \la) \left( S_2^{[1]}(1, \la) + C_2(1, \la) - 2 \right)
\end{equation}
with respect to their multiplicities.
The asymptotic behavior of the eigenvalues is described by the following lemma.

\begin{lem} \label{lem:asympt}
The problem $L$ has a countable set of eigenvalues, which can be numbered as 
$\{ \la_{nk} \}_{n \in \mathbb Z, k = \overline{1, m}} \cup \{ \la_{n0} \}_{n \in \mathbb N}$
(counting with the multiplicities), satisfying
\begin{equation} \label{asymptla}
\arraycolsep=1.4pt\def\arraystretch{1.5}
\left.
\begin{array}{ll}
    \rho_{nk} & := \sqrt{\la_{nk}} = |2 \pi n + \al_k| + \varkappa_n, \quad n \in \mathbb Z, \: k = \overline{1, m}, \\
    \rho_{n0} & := \sqrt{\la_{n0}} = \pi n + \varkappa_n, \quad n \in \mathbb N,
\end{array}
\qquad \right\}
\end{equation}
where 
\begin{equation} \label{alk}
    \al_k \in \left( \frac{(k-1)\pi}{m}, \frac{\bigl(k-\frac{1}{2}\bigr)\pi}{m} \right), \quad k = \overline{1, m}.
\end{equation}
\end{lem}

\begin{proof}
In the case $\sigma_j = 0$, $j = 1, 2$, the characteristic function \eqref{Delta} takes the form
$$
   \Delta_0(\la) = \cos \rho m \frac{\sin \rho}{\rho} + \frac{\sin \rho m}{\rho} (2 \cos \rho - 2), 
$$
where $\rho = \sqrt{\la}$. Note that the function $D(\rho) := \rho \Delta_0(\rho^2)$ is odd and $2\pi$-periodic, so it is sufficient to investigate its zeros 
on $[0, \pi)$. On the one hand, $D(\rho) = \sin \rho P_m(\cos \rho)$, where $P_m(z)$ is a polynomial of degree $m$,
so $D(\rho)$ has no more than $m + 1$ zeros (counting with their multiplicities) on $[0, \pi)$. On the other hand, for $\rho \ne 0$
the equation $D(\rho) = 0$ is equivalent to the following one
\begin{equation} \label{tan}
    \tan \rho m = -\frac{\sin \rho}{2 \cos \rho - 2}.
\end{equation}
\begin{figure}[h!]
\begin{center}
\includegraphics[scale = 0.7]{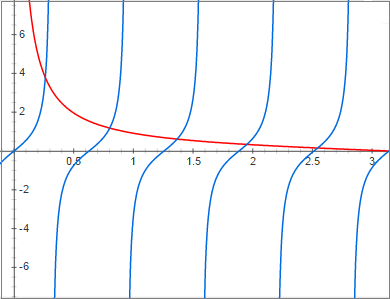}
\end{center}
\caption{Plots for equation \eqref{tan}, $m = 5$}
\label{img:tan}
\end{figure}
From the plots in Figure~\ref{img:tan}, one can easily see that this equation has exactly $m$ simple roots $\al_k$ on the interval $(0, \pi)$, satisfying \eqref{alk}.
Moreover, $D(0) = 0$. Thus, the function $\Delta_0(\la)$ has the zeros 
$$
    \la_{nk}^0 = (2 \pi n + \al_k)^2, \quad n \in \mathbb Z, \: k = \overline{1, m}, \qquad \la_{n0}^0 = (\pi n)^2, \quad n \in \mathbb N. 
$$

Now we turn to the case of nonzero potentials. Using \eqref{CS}, we obtain the relation
$$
   \Delta(\la) = \Delta_0(\la) + \frac{\varkappa_{m+1, odd}(\rho)}{\rho}.
$$
Applying the standard argument, based on Rouche's theorem (see, for example, \cite[Theorem~1.1.3]{FY01}), we arrive at the asymptotic formulas \eqref{asymptla} for the eigenvalues of the
problem $L$.
\end{proof}

\bigskip

{\bf \large 3. Periodic inverse Sturm-Liouville problem}

\bigskip

Inverse spectral problems on graphs with cycles usually generalize the periodic inverse problem on a finite interval.
We describe the periodic problem on a loop $e_2$ in this section, because we need it for statement and solution of our partial inverse problem.

Define 
$$
h(\la) := S_2(1, \la), \quad H(\la) := C_2(1, \la) - S_2^{[1]}(1, \la), \quad d(\la) := S_2^{[1]}(1, \la) + C_2(1, \la) - 2.
$$ 
Denote by $\{ \nu_n \}_{n \in \mathbb N}$
the zeros of the entire function $h(\la)$ and put $\omega_n := \mbox{sign}\, H(\nu_n)$, $\Omega := \{ \om_n \}_{n \in \mathbb N}$. 
The periodic inverse problem is formulated as follows.

\begin{ip} \label{ip:periodic}
Given the functions $h(\la)$, $d(\la)$ and the sequence of signs $\Omega$, construct the potential $\sigma_2$.
\end{ip}

Analogs of Inverse Problem~\ref{ip:periodic} for the case of a regular potential $q_2 \in L_2(0, 1)$ has been studied 
in \cite{Stan70, MO75} (see also paper \cite{Yur08}, where the solution of the periodic problem
has been applied to the inverse problem on a graph). However, the known results can be easily generalized for the case $q_2 \in W_2^{-1}(0, 1)$.
Indeed, it is easy to check that 
\begin{gather} \label{sm3} 
   S_2^{[1]}(1, \nu_n) = (d(\nu_n) + 2 - H(\nu_n))/2, \\
   C_2(1, \la) S_2^{[1]}(1, \la) - C_2^{[1]}(1, \la) S_2(1, \la) \equiv 1.
\end{gather}
Consequently, we have
$$
    H^2(\la) - (d(\la) + 2)^2 = -4 (1 + C_2^{[1]}(1, \la) h(\la)).
$$
Hence
\begin{equation} \label{Hnun}
  	H(\nu_n) = \om_n \sqrt{d(\nu_n)(d(\nu_n) + 2)}.
\end{equation}

% Note that the functions $h(\la)$ and $d(\la)$ do not have common zeros, if and only if $\om_n \ne 0$, $n \in \mathbb N$.
% Indeed, suppose that $\om_n = 0$.
% Since $H(\nu_n) = 0$ and $d(\nu_n) \ne 0$, we get $d(\nu_n) = -2$. In view of \eqref{sm3}, the relation $S_2^{[1]}(1, \nu_n) = 0$  
% holds together with $S_2(1, \nu_n)$ = 0. This leads to a contradiction, hence $\om_n \ne 0$. 
% Thus, if $h(\la)$ and $d(\la)$ do not have common zeros, then $\om_n \ne 0$, $n \in \mathbb N$. The opposite is obvious.

Introduce the norming constants
$$
   \be_n := \int_0^1 S_2^2(x, \nu_n) \, dx. 
$$
Using the standard methods (see \cite[Lemma 1.1.1]{FY01}), one can show that
\begin{equation} \label{ben}
   \be_n = \dfrac{d}{d\la} h(\la)_{|\la = \nu_n} S_2^{[1]}(1, \nu_n). 
\end{equation}
It is proved in \cite{HM03}, that the spectral data $\{ \nu_n, \be_n \}_{n \in \mathbb N}$ uniquely specify
the potential $\sigma_2$, and an algorithm for the reconstruction is provied. 
Thus, Inverse Problem~\ref{ip:periodic} has a unique solution, which can be found by the following algorithm.

\begin{alg} \label{alg:periodic}
Let the functions $h(\la)$, $d(\la)$ and the signs $\Omega$ be given.

\begin{enumerate}
\item Find $\{ \nu_n \}_{n \in \mathbb N}$ as the zeros of $h(\la)$.
\item Calculate $H(\nu_n)$, $n \in \mathbb N$ by \eqref{Hnun}.
\item Find $S_2^{[1]}(1, \nu_n)$ by \eqref{sm3}.
\item Calculate the norming constants $\{ \be_n \}_{n \in \mathbb N}$, using \eqref{ben}.
\item Recover the potential $\sigma_2$ from the spectral data $\{ \nu_n, \be_n \}_{n \in \mathbb N}$, applying the algorithm from~\cite{HM03}. 
\end{enumerate}
\end{alg}

\bigskip

{\large \bf 4. Partial inverse problem}

\bigskip

In this section, we give the statement of the studied partial inverse problem, prove the uniqueness theorem and develop a constructive algorithm for its solution.

Fix a $k = \overline{1, m}$, and denote by $\mathcal I$ the set of indices $\{ (n, k) \colon n \in \mathbb Z \} \cup \{ (n, 0) \colon n \in \mathbb N \}$.
Consider the subspectrum $\Lambda := \{ \la_{nj} \}_{(n, j) \in \mathcal I}$. 
Here and below we assume that the eigenvalues are numbered with respect to their asymptotics according to Lemma~\ref{lem:asympt}.
Note that this numeration is not unique, so a finite number of first eigenvalues in $\Lambda$ can be chosen arbitrarily. 

Impose the following {\bf assumptions}:

\smallskip

($A_1$) All the values in $\Lambda$ are distinct.

($A_2$) All the values in $\Lambda$ are positive.

($A_3$) The functions $h(\la)$ and $d(\la)$ do not have common zeros.

\smallskip

Assumption ($A_1$) is used for simplicity, the case of multiple eigenvalues require some tecnical modifications
(see discussion in \cite{Bond17-1}). Assumption ($A_2$) can be achieved by a shift of the spectrum.
Assumption ($A_3$) is the only principal one. One can easily check, that ($A_3$) is equivalent to the condition
$\om_n \ne 0$, $n \in \mathbb N$.

Under assumptions~($A_1$)--($A_3$), we study the following partial inverse problem.

\begin{ip} \label{ip:part}
Given the potential $\sigma_1$, the subspectrum $\Lambda$ and the signs $\Omega$,
find the potential $\sigma_2$.
\end{ip}

Proceed to the solution of the formulated problem.
Using relations \eqref{CS}, we get
\begin{equation} \label{intKN}
\arraycolsep=1.4pt\def\arraystretch{1.5}
\left.
\begin{array}{l}
   S_2(1, \la) = \frac{\sin \rho}{\rho} + \frac{1}{\rho} \int_0^1 K(t) \sin \rho t \, dt, \\
   d(\la) = 2 \cos \rho + 2 \int_0^1 N(t) \cos \rho t \, dt - 2,
\end{array}
\qquad \right\}
\end{equation}
where $K$ and $N$ are some real-valued functions from $L_2(0, 1)$.
Substituting \eqref{intKN} into \eqref{Delta}, we derive the relation
\begin{equation} \label{relKN}
    \int_0^1 K(t) a_{nj} \sin \rho_{nj} t \, dt + \int_0^1 N(t) b_{nj} \cos \rho_{nj} t \, dt = f_{nj}, \quad (n, j) \in \mathcal I,
\end{equation}
where
\begin{equation} \label{defabf}
   a_{nj} = S_1^{[1]}(m, \la_{nj}), \quad b_{nj} = 2 \rho_{nj} S_1(m, \la_{nj}), \quad
   f_{nj} = -a_{nj} \sin \rho_{nj} - b_{nj} (\cos \rho_{nj} - 1).
\end{equation}
Introduce the real Hilbert space $\mathcal H := L_2(0, 1) \oplus L_2(0, 1)$ with the scalar product
$$
   (g, h)_{\mathcal H} = \int_0^1 ( g_1(t) h_1(t) + g_2(t) h_2(t) ) \, dt, \quad g, h \in \mathcal H, \quad 
   g = \begin{bmatrix} g_1 \\ g_2 \end{bmatrix}, \: h = \begin{bmatrix} h_1 \\ h_2 \end{bmatrix}.
$$
Obviously, the vector functions
\begin{equation} \label{defv}
   f(t) := \begin{bmatrix} K(t) \\ N(t) \end{bmatrix}, \quad 
   v_{nj}(t) := \begin{bmatrix}a_{nj} \sin \rho_{nj} t \\ b_{nj} \cos \rho_{nj} t \end{bmatrix}, \quad (n, j) \in \mathcal I,
\end{equation}
belong to $\mathcal H$, and relation \eqref{relKN} can be rewritten in the form
\begin{equation} \label{scal}
 (f, v_{nj})_{\mathcal H} = f_{nj}, \quad (n, j) \in \mathcal I.
\end{equation}

\begin{lem} \label{lem:complete}
The system of vector functions $\mathcal V := \{ v_{nj} \}_{(n, j) \in \mathcal I}$ is complete in $\mathcal H$.
\end{lem}

\begin{proof}
Suppose $w_1, w_2 \in L_2(0, 1)$ are such functions, that
\begin{equation} \label{sm1}
    \int_0^1 \left( w_1(t) a_{nj} \sin \rho_{nj} t  + w_2(t) b_{nj} \cos \rho_{nj} t \right) \, dt = 0, \quad (n, j) \in \mathcal I.
\end{equation}
Let $S_1(m, \la_{nj}) \ne 0$ for some $(n, j) \in \mathcal I$. By assumption ($A_3$), we have 
$S_2(1, \la_{nj}) \ne 0$. Therefore, using \eqref{Delta}, \eqref{defabf} and taking assumption ($A_2$) into account, we get
$$
    a_{nj} = -\dfrac{b_{nj} d(\la_{nj})}{2 \rho_{nj} S_2(1, \la_{nj})}.
$$
Substituting this relation into \eqref{sm1}, we obtain
\begin{equation} \label{sm2}
\int_0^1 \left( w_1(t) d(\la_{nj}) \dfrac{\sin \rho_{nj} t}{\rho_{nj}} - 2 w_2(t) S_2(1, \la_{nj}) \cos \rho_{nj} t \right) \, dt = 0. 
\end{equation} 
In the case $S_1(m, \la_{nj}) = 0$ we have $S_1^{[1]}(m, \la_{nj}) \ne 0$. In view of \eqref{Delta}, $S_2(1, \la_{nj}) = 0$.
Assumption ($A_3$) yields $d(\la_{nj}) \ne 0$. Consequently, the relation \eqref{sm1} implies \eqref{sm2}. Thus, \eqref{sm2} holds 
for all $(n, j) \in \mathcal I$. Hence the entire function
\begin{equation} \label{defW}
   W(\la) := \int_0^1 \left( w_1(t) d(\la) \frac{\sin \rho t}{\rho} - 2 w_2(t) S_2(1, \la) \cos \rho t \right) \, dt
\end{equation}
has zeros $\Lambda$. Clearly, together with \eqref{intKN}, we get
\begin{equation} \label{asymptW}
W(\la) = O\left( |\rho|^{-1} \exp(2 |\mbox{Im}\, \rho|) \right), \quad |\rho| \to \iy.
\end{equation}
Taking assumption ($A_2$) into account, construct the infinite product
$$
D(\la) := \prod_{(n, j) \in \mathcal I} \left( 1 - \frac{\la}{\la_{nj}} \right).
$$
In view of assumption ($A_1$), the function $\dfrac{W(\la)}{D(\la)}$ is entire.
According to the asymptotic formulas \eqref{asymptla}, the function $D(\la)$ can be represented in the following form (see \cite[Appendix~B]{Bond17-preprint}):
\begin{equation} \label{relD}
    D(\la) = C(\cos \rho - \cos \al_k) \frac{\sin \rho}{\rho} + \frac{\varkappa_{2, odd}(\rho)}{\rho}, 
\end{equation}
where $C$ is a nonzero constant. Moreover, one has the following estimate 
$$
  	|D(\rho^2)| \ge C |\rho|^{-1} \exp(2 |\mbox{Im} \, \rho|) , \quad \eps < \arg \rho < \pi - \eps, \: |\rho| \ge \rho^*
$$
for some positive $\eps$ and $\rho^*$. Together with \eqref{asymptW} it yields
$$
   \frac{W(\la)}{D(\la)} = O(1), \quad \la = \rho^2, \: \eps < \arg \rho < \pi - \eps, \: |\rho| \ge \rho^*. 
$$
By Phragmen-Lindel\"of's and Liouville's theorems we get $W(\la) \equiv C D(\la)$. Using \eqref{defW}, one can show that
$\rho W(\rho^2) \in B_{2,2}$ (as a function of the variable $\rho$). However, relation \eqref{relD} implies
$\rho D(\rho^2) \not \in B_{2, 2}$. Hence $C \equiv 0$ and $W(\la) \equiv 0$.

Recall that $\{ \nu_n \}_{n \in \mathbb N}$ are the zeros of $S_2(1, \la)$. Assumption ($A_3$) requires $d(\nu_n) \ne 0$, $n \in \mathbb N$.
Consequently, it follows from \eqref{defW}, that
$$
   	\int_0^1 w_1(t) \frac{\sin \sqrt{\nu_n} t}{\sqrt{\nu_n}} \, dt = 0, \quad n \in \mathbb N.
$$
Note that $\{ \nu_n \}_{n \in \mathbb N}$ are the eigenvalues of the boundary value problem 
$$
    \ell_2 y_2 = \la y_2, \quad y_2(0) = y_2(1) = 0,
$$
therefore $\sqrt{\nu_n} = \pi n + \varkappa_n$, $n \in \mathbb N$ (see \cite{Sav01}).
This asymptotic relation implies that the system $\left\{ (\sqrt{\nu_n})^{-1} \sin \sqrt{\nu_n} t \right\}_{n \in \mathbb N}$
is complete in $L_2(0, 1)$ (see, for example, \cite{HV01}). Hence $w_1 = 0$. Then we conclude from \eqref{asymptW} and $W(\la) \equiv 0$, that $w_2 = 0$.
Thus, the system $\mathcal V$ is complete in $\mathcal H$.
\end{proof}

Relying on Lemma~\ref{lem:complete}, we shall prove the uniqueness theorem for the solution of Inverse Problem~\ref{ip:part}.
Along with the boundary value problem $L$, consider the problem $\tilde L$ of the same form, but with different potentials
$\tilde \sigma_j \in L_2(0, l_j)$, $j = 1, 2$. We agree that if a certain symbol $\ga$ denotes an object related to $L$, the corresponding symbol
$\tilde \ga$ denotes an analogous object related to $\tilde L$.

\begin{thm} \label{thm:uniq}
Suppose that the boundary value problems $L$ and $\tilde L$ together with their subspectra $\Lambda$ and $\tilde \Lambda$
of the form described above satisfy assumptions ($A_1$)--($A_3$), and $\sigma_1(x) = \tilde \sigma_1(x)$ a.e. on $(0, m)$,
$\Lambda = \tilde \Lambda$, $\Omega = \tilde \Omega$. Then $\sigma_2(x) = \tilde \sigma_2(x)$ a.e. on $(0, 1)$. Thus, Inverse Problem~\ref{ip:part}
has a unique solution.
\end{thm}

\begin{proof}
The relation $\sigma_1(x) = \tilde \sigma_1(x)$ a.e. on $(0, m)$ implies $S_1(m, \la) \equiv \tilde S_1(m, \la)$,
$S_1^{[1]}(m, \la) \equiv \tilde S_1^{[1]}(m, \la)$. In view of \eqref{defabf}, \eqref{defv} and $\Lambda = \tilde \Lambda$,
we have $v_{nj} = \tilde v_{nj}$ in $\mathcal H$ and $f_{nj} = \tilde f_{nj}$ for $(n, j) \in \mathcal I$. Since by Lemma~\ref{lem:complete}
the system $\mathcal V$ is complete in $\mathcal H$, we conclude from \eqref{scal}, that $K(t) = \tilde K(t)$ and $N(t) = \tilde N(t)$
a.e. on $(0, 1)$. Then relation \eqref{intKN} yields $S_2(1, \la) \equiv \tilde S_2(1, \la)$, $d(\la) \equiv \tilde d(\la)$.
In addition, we have $\Omega = \tilde \Omega$, so $\sigma_2(x) = \tilde \sigma_2(x)$ follows from the uniqueness of the solution
of periodic Inverse Problem~\ref{ip:periodic}.
\end{proof}

\begin{thm} \label{thm:Riesz}
The system of vector functions $\mathcal V$ is a Riesz basis in $\mathcal H$.
\end{thm}

\begin{proof}
Using \eqref{asymptla} and \eqref{defabf}, we get
$$
   a_{nj} = \cos \rho_{nj} m + \varkappa_n, \quad b_{nj} = 2 \sin \rho_{nj} m + \varkappa_n, \quad (n, j) \in \mathcal I.
$$
Consequently, we have $\{ \| v_{nj} - v_{nj}^0 \|_{\mathcal H} \}_{(n, j) \in \mathcal I} \in l_2$, where
$$
    v_{nk}^0(t) = \begin{bmatrix} \cos \al_k m \sin |2 \pi n + \al_k| t \\ 2 \sin \al_k m \cos |2 \pi n + \al_k| t \end{bmatrix}, \: n \in \mathbb Z,
    \quad
    v_{n0}^0(t) = \begin{bmatrix} \sin \pi n t \\ 0 \end{bmatrix}, \: n \in \mathbb N.
$$
Note that \eqref{alk} implies $\cos \al_k m \ne 0$, $\sin \al_k m \ne 0$.

Let us show that the system $\mathcal V^0 := \{ v_{nj}^0 \}_{(n, j) \in \mathcal I}$ is a Riesz basis in $\mathcal H$.
It follows from the results of \cite[Appendix~A]{Bond17-preprint}, that the systems $\{ \sin (2 \pi n + \al_k) t \}_{n \in \mathbb Z}$
and $\{ \cos (2 \pi n + \al_k) t \}_{n \in \mathbb Z}$ are Riesz bases in $L_2(0, 1)$.
Consider the linear operator $A \colon \mathcal H \to \mathcal H$, defined as follows.
$$
  	A v = A \begin{bmatrix} v_1 \\ v_2 \end{bmatrix} = \begin{bmatrix} v_1 - \frac{1}{2} \cot(\al_k m) g(v_2) \\ v_2 \end{bmatrix}, 
  	\quad v \in \mathcal H,
$$
where
$$
  	g(u)(t) = \sum_{n \in \mathbb Z} c_n(u) \sin |2 \pi n + \al_k| t, \quad u(t) = \sum_{n \in \mathbb Z} c_n(u) \cos |2 \pi n + \al_k| t,  
$$
i.e. $c_n(u)$ are the coordinates of the function $u \in L_2(0, 1)$ with respect to the Riesz basis $\{ \cos |2 \pi n + \al_k| t \}_{n \in \mathbb Z}$.
It follows from the Riesz-basis property, that there exist positive constants $C_1$ and $C_2$ such that
$$
    C_1 \| u \|_{L_2} \le \| g(u) \|_{L_2} \le C_2 \| u \|_{L_2}.
$$
Consequently, the operator $A$ and its inverse:
$$
    A^{-1} v = A \begin{bmatrix} v_1 \\ v_2 \end{bmatrix} = \begin{bmatrix} v_1 + \frac{1}{2} \cot(\al_k m) g(v_2) \\ v_2 \end{bmatrix}, 
  	\quad v \in \mathcal H, 
$$ 
are bounded in $\mathcal H$. Note that the operator $A$ transforms the sequence $\mathcal V^0$ into a Riesz basis
in $\mathcal H$:
$$
   (A v_{nk}^0)(t) = 2 \sin \al_k m \begin{bmatrix} 0 \\ \cos (2 \pi n + \al_k) t \end{bmatrix}, \: n \in \mathbb Z, \quad
   (A v_{n0}^0)(t) = \begin{bmatrix} \sin \pi n t \\ 0 \end{bmatrix}, \: n \in \mathbb N.  
$$
Hence the system $\mathcal V^0$ is also a Riesz basis. 

Since the system $\mathcal V$ is complete by Lemma~\ref{lem:complete} and
$l_2$-close to the Riesz basis $\mathcal V^0$, we conclude that $\mathcal V$ is a Riesz basis in $\mathcal H$. 
\end{proof}

Recovering the vector function $f$ from its coordinates with respect to the Riesz basis, one can solve Inverse Problem~\ref{ip:part}
by the following algorithm.

\begin{alg} \label{alg:part}
Let the potential $\sigma_1$, the eigenvalues $\Lambda$ and the signs $\Omega$ be given.

\begin{enumerate}
\item Construct the functions $S_1(m, \la)$ and $S_1^{[1]}(m, \la)$.
\item Find the vector functions $v_{nj}$ and the numbers $f_{nj}$, using \eqref{defabf} and \eqref{defv}.
\item Construct the vector function $f$ by its coordinates with respect to the Riesz basis (see \eqref{scal}), i.e. find the functions 
$K(t)$ and $N(t)$.
\item Construct the functions $h(\la) := S_2(1, \la)$ and $d(\la)$ by \eqref{intKN}.
\item Recover the potential $\sigma_1$ from $h(\la)$, $d(\la)$ and $\Omega$, using Algorithm~\ref{alg:periodic}. 
\end{enumerate}

\end{alg}

{\bf Acknowledgment.} 
The author C.-F.~Yang was supported in part by the National Natural Science Foundation of
China (11171152, 11611530682 and 91538108) and by the Natural Science Foundation of
the Jiangsu Province of China (BK 20141392).
The author N.~P.~Bondarenko was supported by the Russian Federation
President Grant MK-686.2017.1, by Grant 1.1660.2017/4.6 of the Russian
Ministry of Education and Science, and by Grants
16-01-00015, 17-51-53180 of the Russian Foundation for Basic Research.

\medskip

\noindent Chuan-Fu Yang \\
Department of Applied Mathematics, Nanjing University of Sciences and Technology, \\
Nanjing, 210094, Jiangsu, China, \\
email: {\it chuanfuyang@njust.edu.cn}

\medskip

\noindent Natalia Pavlovna Bondarenko \\
1. Department of Applied Mathematics, Samara National Research University, \\
Moskovskoye Shosse 34, Samara 443086, Russia, \\
2. Department of Mechanics and Mathematics, Saratov State University, \\
Astrakhanskaya 83, Saratov 410012, Russia, \\
e-mail: {\it BondarenkoNP@info.sgu.ru}

\end{document}